\documentclass[12pt,draft]{amsart}
\usepackage{amsthm,amsmath,amsfonts,amssymb,amscd,mathrsfs,hyperref,cleveref,geometry,bbm}
\usepackage{latexsym}
\usepackage{graphicx}
\usepackage{amsaddr}
\usepackage{cite}

\theoremstyle{plain}
\newtheorem{thm}{Theorem}
\newtheorem{lem}{Lemma}

\theoremstyle{definition}
\newtheorem*{dfn}{Definition}

\crefname{thm}{theorem}{theorems}
\crefname{lem}{lemma}{lemmas}

\DeclareMathOperator{\E}{\mathbb{E}}
\DeclareMathOperator{\Prob}{Pr}

\newcommand{\ds}{\displaystyle}
\newcommand{\bs}{\boldsymbol}
\newcommand{\mb}{\mathbb}
\newcommand{\mc}{\mathcal}

\renewcommand{\mod}{\operatorname{mod}}

\def \a{\alpha} \def \b{\beta}  \def \e{\varepsilon} \def \g{\gamma}  \def \l{\lambda} \def \s{\sigma} \def \t{\theta} 
\newcommand{\rz}[1]{\displaystyle\stackrel{#1}{\phantom{}}}

\numberwithin{equation}{section}
\begin{document}
\title{Clusters of primes with square-free translates}
\author[R.\,C. Baker]{Roger C. Baker}
\address{Department of Mathematics\\Brigham Young University\\Provo, UT 84602, USA}
\email{baker@math.byu.edu}

\author[P. Pollack]{Paul Pollack}
\address{Department of Mathematics\\University of Georgia\\Athens, GA 30602, USA}
\email{pollack@uga.edu}

\begin{abstract} Let $\mc R$ be a finite set of integers satisfying appropriate local conditions. We show the existence of long clusters of primes $p$ in bounded length intervals with $p-b$ squarefree for all $b \in \mc R$. Moreover, we can enforce that the primes $p$ in our cluster satisfy any one of the following conditions: (1) $p$ lies in a short interval $[N, N+N^{\frac{7}{12}+\e}]$, (2) $p$ belongs to a given inhomogeneous Beatty sequence, (3) with $c \in (\frac{8}{9},1)$ fixed, $p^c$ lies in a prescribed interval mod $1$ of length $p^{-1+c+\e}$.
\end{abstract}

\maketitle

\section{Introduction}

In important recent work, Maynard \cite{may} has shown that for a given integer $t \ge 2$ and sufficiently large $N$, there is a set $\mc S$ of $t$ primes in $[N, 2N)$ with diameter
 \[D(\mc S):= \max_{n\in \mc S}\, n -
 \min_{n \in \mc S}\, n \ll t^3 e^{4t}.\]
In \cite{bakzhao}, the authors adapted \cite{may} to obtain similar results for primes in a subset $\mc A$ of $[N, 2N)$, subject to arithmetic regularity conditions on $\mc A$.

In the present paper, we impose the further condition on $\mc S$ that (for a given nonzero integer $b$) $p-b$ is squarefree for each $p$ in $\mc S$. A little more generally, we treat the differences $p-b$ $(b \in \mc R)$, where $\mc R$ is a {\it reasonable} set.

 \begin{dfn}
A set $\{b_1, \ldots, b_r\}$ of nonzero integers is {\it reasonable} if for every prime $p$ there is an integer $v$, $p\nmid v$, with
 \[b_\ell \not\equiv v\pmod{p^2}\quad (\ell=1,\ldots, r).\]
 \end{dfn}

A little thought shows that,~if there are infinitely many primes $p$ with $p-b_1, \ldots, p-b_r$ all squarefree, then $\{b_1, \ldots, b_r\}$ is reasonable. From now on, let $\mc R$ be a fixed reasonable set.

In order to state our general result we require some notation. We suppose that $t$ is fixed, that $N$ is sufficiently large (in particular, $N \ge C(t)$) and write $\mc L = \log N$,
 \[D_0 = \frac{\log \mc L}{\log\log \mc L}.\]
We denote by $\tau(n)$ and $\tau_k(n)$ the usual divisor functions. Let $\e$ be a sufficiently small positive number. Let
 \[P(z) = \prod_{p<z} p\]
(we reserve the symbol $p$ for primes). Let $\mb P$ denote the set of primes. Let $X(E;n)$ denote the indicator function of a set $E$.

For a smooth function $F$ supported on
 \[\mc E_k := \Bigg\{(x_1,\ldots, x_k) \in [0,1]^k:
 \sum_{j=1}^k\, x_j \le 1\Bigg\},\]
let
 \[I_k(F) := \int_0^1 \ldots \int_0^1 F(t_1, \ldots,
 t_k)^2\, dt_1\ldots dt_k\]
and
 \[J_k^{(m)} (F) := \int_0^1 \ldots \int_0^1 \left(
 \int_0^{1} F(t_1,\ldots, t_k)\, dt_m\right)^2
 dt_1\ldots dt_{m-1} dt_{m+1}\ldots dt_k\]
for $m=1, \ldots, k$. Let $\mc F_k$ denote the set of smooth functions supported on $\mc E_k$ and for which $I_k(F)$ and each $J_k^{(m)}(F)$ is positive and let
 \[M_k = \sup_{F \in \mc F_k}\
 \frac{\sum\limits_{m=1}^k J_k^{(m)}(F)}{I_k(F)}.\]

 \begin{dfn}
(i) A set $\mc H_k = \{h_1, \ldots, h_k\}$ of integers $0 \le h_1 < \cdots < h_k$, is {\it admissible} if for every prime $p$, there is an integer $v_p$ such that $v_p \not\equiv h \pmod p$ for all $h \in \mc H_k$.

(ii) Let $\mc H_k= \{h_1, \ldots, h_k\}$ be admissible. Write
 \[K = (r+1)k+1, \quad P = P(K).\]
Suppose that
 \begin{align}
&h_m \equiv 0 \pmod{P^2} \quad (m=1,\ldots, k)\label{eq1.1}\\[2mm]
&h_i - h_j + b_\ell \ne 0\quad (i,j=1, \ldots, k; \ell = 1, \ldots, r). \label{eq1.2}
 \end{align}
We say that $\mc H_k$ is {\it compatible with} $\mc R$.
 \end{dfn}

 \begin{thm}\label{thm1}
Let $\mc R$ be a reasonable set and $\mc H_k$ an admissible set compatible with $\mc R$. Let $N \in \mb N$, $N > C_0(\mc R, \mc H_k$). Let $\mc A \subset [N, N+M) \cap \mb Z$ where $N^{1/2} \mc L^{18k} \le M \le N$. Let $\t$ be a constant, $0 < \t < 3/4$. Let $Y$ be a positive number,
 \begin{equation}\label{eq1.3}
N^{1/4} \max(N^\t, \mc L^{9k} M^{1/2}) \ll Y \le M.
 \end{equation}

Let
 \[V(q) = \max_a \Bigg|\sum_{n\,\equiv\, a\, (\mod q)}
 X(\mc A; n) - \frac Yq\Bigg|.\]
Suppose that, for
  \begin{equation}\label{eq1.4}
1 \le d \le (MY^{-1})^4 \max(\mc L^{36k}, N^{4\t} M^{-2})
 \end{equation}
we have
 \begin{equation}\label{eq1.5}
\sum_{\substack{q \le N^\t\\
(q,d)=1}} \mu^2(q) \tau_{3k}(q) V(dq) \ll Y\mc L^{-k-\e} d^{-1}.
 \end{equation}
Suppose there is a function $\rho(n)$ defined on $[N, 2N) \cap \mb Z$  such that
 \begin{equation}\label{eq1.6}
X(\mb P; n) \ge \rho(n)
 \end{equation}
for $n\in[N,2N)$, and positive numbers $Y_{m}$,
 \begin{equation}\label{eq1.7}
Y_{m} = Y(b_{m} + o(1))\mc L^{-1}\quad (1 \le m \le k)
 \end{equation}
where
 \begin{equation}\label{eq1.8}
b_m \ge b > 0 \ \ (1 \le m \le k).
 \end{equation}
Suppose that $\rho(n) = 0$ unless $(n, P(N^{\t/2}))=1$, and
 \begin{align}\label{eq1.9}
\sum_{q\le N^\t}\ \mu^2(q) \tau_{3k}(q) \max_{(a,q)=1} \Bigg|\sum_{n\,\equiv\, a\ (\mod q)} &\rho(n) X((\mc A + h_m) \cap \mc A; n)- \frac{Y_{m}}{\phi(q)}\Bigg| \ll Y\mc L^{-k-\e}.
 \end{align}
Finally, suppose that
 \begin{equation}\label{eq1.10}
M_k > \frac{2t-2}{b\t}.
 \end{equation}
Then there is a set $\mc S$ of $t$ primes in $\mc A$ such that $p-b$ is squarefree $(p \in \mc S, b \in \mc R)$ and
 \[D(\mc S) \le h_k - h_1.\]

If $Y > N^{1/2+\e}$, the assertion of the theorem is also valid with \eqref{eq1.4} replaced by
 \begin{equation}\label{eq1.11}
1 \le d \le (MY^{-1})^2 N^{2\e}.
 \end{equation}
  \end{thm}

Comparing Theorem 1 of \cite{bakzhao}, where there is no requirement on squarefree translates of $p$ $(p \in \mc S)$, the difference in hypotheses on $\mc A$ is that \eqref{eq1.5} is required only for $d=1$ in \cite{bakzhao}.

This is a convenient point to note that
 \begin{equation}\label{eq1.12}
M_k > \log k - C_1
 \end{equation}
(\kern-4pt\cite[Theorem 3.9]{poly}). Here $C_1, C_2, \ldots$ are positive absolute constants. Constants implied by `$\ll$' are permitted to depend on $\mc R, \mc H_k$, and $\e$.

We now give three applications of \Cref{thm1}.

 \begin{thm}\label{thm2}
Let $\a$ be an irrational number, $\a > 1$. Let $\b$ be real. Let $v$ be a sufficiently large integer and $u$ an integer with $(u, v) = 1$,
 \[\left|\a - \frac uv\right| < \frac 1{v^2}.\]
There is a set $\mc S$ of $t$ primes of the form $\lfloor \a n + \b\rfloor$ in $[v^2, 2v^2)$ such that $p-b$ is squarefree $(p \in \mc S, \, b \in \mc R)$ and
  \begin{equation}\label{eq1.13}
D(\mc S) < \exp(C_2\, \a r \exp(7.743t)).
 \end{equation}
(We write $\lfloor \ldots \rfloor$ for integer part and $\{\ldots\}$ for fractional part.)
 \end{thm}

 \begin{thm}\label{thm3}
Let $7/12 < \phi < 1$ and write
 \[\psi = \begin{cases}
 \phi - 11/20 -\e & (7/12 < \t < 3/5),\\[2mm]
 \phi - 1/2 -\e & (\phi \ge 3/5).\end{cases}\]
For all sufficiently large $N$, there is a set $\mc S$ of $t$ primes in $\mc A = [N, N+N^\phi]$ such that $p-b$ is squarefree $(p \in \mc S, \, b \in \mc R)$ and
 \[D(\mc S) < \exp\left(C_3\, r \exp\left(
 \frac{2t}\psi\right)\right).\]
 \end{thm}

 \begin{thm}\label{thm4}
Fix $c$ in $(8/9, 1)$ and real $\b$. Let $0 < \psi < (9c-8)/6$. Let
 \[\mc A = \{n\in [N,2N) : \{n^c -\b\}
 < N^{-1+c+\e}\}.\]
For all sufficiently large $N$, there is a set $\mc S$ of $t$ primes in $\mc A$ such that $p-b$ is squarefree $(p\in \mc S, \, b \in \mc R)$ and
 \[D(\mc S) < \exp\left(C_4\, rt \exp
 \left(\frac{2t}{\psi}\right)\right).\]
 \end{thm}

In Theorems \ref{thm3} and \ref{thm4}, $\mc A$ is relatively small in cardinality compared to $N$. There are rather few examples of this type for which $\mc A \cap (\mc A + h_m)$ has as many primes as $\mc A$ in order of magnitude, which we need for \eqref{eq1.9}.

\section{Proof of \Cref{thm1}}\label{sec2}

Let
 \[W_1 = \prod_{p\le K} \, p^2\,
 \prod_{K < p \le D_0} \, p \ , \quad R = N^{\t/2-\e}.\]
In our proof of \Cref{thm1}, we use weights $y_{\bs r}$ and $\l_{\bs r}$ defined much as in $[7,4]$ by: for $\bs r = (r_1, \ldots, r_k) \in \mb N^k$, $y_{\bs r} = \l_{\bs r} = 0$ unless
 \begin{equation}\label{eq2.1}
\left(\prod_{i=1}^k r_i, W_1\right) = 1, \ \mu^2\left(\prod_{i=1}^k r_i\right) = 1.
 \end{equation}
If \eqref{eq2.1} holds, we take
 \begin{equation}\label{eq2.2}
y_{\bs r} = F\left(\frac{\log r_1}{\log R}, \ldots, \frac{\log r_k}{\log R}\right)
 \end{equation}
where $F \in \mc F_k$ has
 \begin{equation}\label{eq2.3}
\sum_{m=1}^k J_k^{(m)}(F) > (M_k - \e) I_k(F) > \left(\frac{2t-2}{b\t}\right) I_k(F).
 \end{equation}
Now $\l_{\bs d}$ is defined by
 \begin{equation}\label{eq2.4}
\l_{\bs d} = \prod_{i=1}^k \mu(d_i)d_i \sum_{\substack{\bs r\\
d_i \mid r_i \, \forall_i}} \ \frac{y_{\bs r}}{\prod\limits_{i=1}^k \phi(r_i)}
 \end{equation}
when \eqref{eq2.1} holds. Note that
 \begin{equation}\label{eq2.5}
\l_{\bs d} \ll \mc L^k
 \end{equation}
(\kern-4pt\cite[(5.9)]{may}).

We now show that there is an integer $\nu_0$ with
 \begin{gather}
(\nu_0 + h_m, W_1) = 1 \quad (1 \le m \le k)\label{eq2.6}\\[1mm]
p^2 \nmid \nu_0 + h_m - b_\ell \quad (p \le K, \ 1 \le \ell \le r, \ 1 \le m \le k)\label{eq2.7}\\
\intertext{and}
p \nmid \nu_0 + h_m - b_\ell \quad (K < p \le D_0, \ 1 \le \ell \le r, \ 1 \le m \le k).\label{eq2.8}
 \end{gather}
By the Chinese remainder theorem, it suffices to specify $\nu_0 \pmod{p^2}$ for $p \le K$ and $\nu_0\, (\mod p)$ for $K < p \le D_0$. We use $h_j \equiv 0\pmod{p^2}$ $(p \le K)$. The property \eqref{eq2.6} reduces to
 \begin{equation}\label{eq2.9}
\nu_0 \not\equiv 0 \pmod p \quad (p \le K)
 \end{equation}
and
 \begin{equation}\label{eq2.10}
\nu_0 + h_m \not\equiv 0 \pmod p \quad (K < p \le D_0, \ 1 \le m \le k).
 \end{equation}
We define $b_0 = 0$. Now \eqref{eq2.7}, \eqref{eq2.8}, \eqref{eq2.9}, \eqref{eq2.10} can be rewritten as
 \begin{gather}
\nu_0 \not\equiv 0 \pmod p,\ \nu_0 \not\equiv b_\ell \pmod{p^2} \ \ (p \le K, 1 \le \ell \le r),\label{eq2.11}\\[1mm]
\nu_0 + h_m - b_\ell \not\equiv 0\ (\mod p) \ (K < p \le D_0, 0 \le \ell \le r, 1 \le m \le k).\label{eq2.12}
 \end{gather}

For \eqref{eq2.11}, we select $\nu_0$ in a reduced residue class $(\mod p^2)$ not occupied by $b_\ell$ $(1 \le \ell \le r)$. For \eqref{eq2.11}, we observe that $\nu_0$ can be chosen from the $p-1$ reduced residue classes $(\mod p)$, avoiding at most $(r+1)k$ classes, since $p-1 > (r+1)k$.

We now define weights $w_n$. For $n\equiv \nu_0$ $(\mod W_1)$, let
 \[w_n = \Bigg(\sum_{d_i\mid n+h_i\,\forall i}
 \l_{\bs d}\Bigg)^2.\]
For other $n\in \mb N$, let $w_n = 0$. Let
 \begin{gather*}
S_1 = \sum_{\substack{N\le n < 2N\\
n\in \mc A}} w_n,\\[2mm]
S_2(m) = \sum_{\substack{N\le n < 2N\\
n\in \mc A \cap (\mc A - h_m)}}w_n \rho(n+h_m).
 \end{gather*}

Exactly as in the proof of \cite[Proposition 1]{bakzhao} with $q_0 = 1$, $W_2 = W_1$, we find that
 \begin{gather}
S_1 = \frac{(1 + o(1))\phi(W_1)^k Y(\log R)^k I_k(F)}{W_1^{k+1}}\label{eq2.13}\\
\intertext{and}
S_2(m) = \frac{(1+ o(1)) b_{m} \phi(W_1)^{k} Y(\log R)^{k+1} J_k^{(m)}(F)}{W_1^{k+1}\mc L}\label{eq2.14}
 \end{gather}
as $N\to \infty$. (The value of $W_1$ in \cite{bakzhao} is $\prod\limits_{p\le D_0}\, p$, but this does not affect the proof.)

Exactly as in \cite{bakzhao} following the statement of Proposition 2, we derive from \eqref{eq2.13}, \eqref{eq2.14}, \eqref{eq2.3}, the inequality
 \begin{equation}\label{eq2.15}
\sum_{m=1}^k\ \sum_{n\in\mc A}\, w_n X(\mb P \cap \mc A, \, n + h_m) > (t-1+\e) \sum_{n\in\mc A}\, w_n.
 \end{equation}

We introduce a probability measure on $\mc A$ defined by
 \[ \Prob\{n\} = \frac{w_n}{S_1}\ (n \in \mc A).\]
Writing $\E[\cdot]$ for expectation, \eqref{eq2.15} becomes
 \[ \E\Bigg[\sum_{m=1}^k X(\mb P \cap \mc A; \
 n + h_m)\Bigg] > t - 1 + \e.\]
It is easy to deduce that
 \[ \Prob\Bigg(\sum_{m=1}^k X(\mb P \cap \mc A; \
 n + h_m)\ge t\Bigg) > \frac \e k.\]

It now suffices to show for fixed $m\in \{1, \ldots, k\}$ and $\ell \in \{1, \ldots, r\}$ that
 \begin{equation}\label{eq2.16}
\Prob(n+h_m - b_\ell \text{ is not square-free}) \ll D_0^{-1}.
 \end{equation}
For then there is a probability greater than $\e/2k$ that an integer $n$ in $\mc A$ has the property that at least $t$ of $n+h_1,\ldots, n+h_k$ are primes in $\mc A$ for which all translates $n+h_m - b_\ell$ $(1 \le m \le k, \, 1 \le \ell \le r)$ are square-free.

The upper bound
 \begin{equation}\label{eq2.17}
\sum_{\substack{
 N \le n < N + M\\
 n \equiv \nu_0\ (\mod W_1)}} w_n^2 \ll \mc L^{19k}
 \frac M{W_1} + N^{2\t}
 \end{equation}
can be proved in exactly the same way as \cite[(3.10)]{poll}.

Let
 \[\Omega(p) = \sum \{w_n : n \in \mc A, \,
 p^2\, |\, n + h_m - b_\ell\}\]
and
 \[B = (MY^{-1})^2 \max (\mc L^{18k}, N^{2\t} M^{-1}).\]
Clearly
 \begin{align*}
\Prob(n + h_m - b_\ell & \text{ is not square-free})
\le \frac 1{S_1} \Bigg(\sum_{p \le B} \Omega(p) +
\sum_{\substack{n\in \mc A\\
p^2\,|\, n+h_m-b_\ell \, (\text{some } p > B)}} w_n\Bigg).
 \end{align*}
To obtain \eqref{eq2.16} we need only show that
 \begin{equation}\label{eq2.18}
\sum_{p\le B} \Omega(p) \ll \frac{\phi(W_1)^k Y\mc L^k}{W_1^{k+1} D_0}
 \end{equation}
and
 \begin{equation}\label{eq2.19}
\sum_{\substack{n\in \mc A\\
p^2\,|\, n+h_m-b_\ell \ (\text{some } p > B)}} w_n \ll \frac{\phi(W_1)^kY\mc L^k}{W_1^{k+1} D_0}
 \end{equation}

From \eqref{eq2.6}--\eqref{eq2.8}, $\Omega(p)=0$ for $p \le D_0$. Take $D_0 < p \le B$. We have
 \begin{equation}\label{eq2.20}
\Omega(p) = \sum_{\bs d, \bs e} \l_{\bs d} \l_{\bs e} \sum_{\substack{n\in \mc A\\
n\equiv \nu_0 \ (\mod W_1)\\
n\equiv b_\ell - h_m \ (\mod p^2)\\
n\equiv -h_i \ (\mod\,[d_i, e_i])\, \forall i}}1.
 \end{equation}

Fix $\bs d$, $\bs e$ with $\l_{\bs d} \l_{\bs e} \ne 0$. The inner sum in \eqref{eq2.20} is empty if $(d_i, e_j) > 1$ for a pair $i$, $j$ with $i\ne j$ (compare \cite[\S2]{bakzhao}).  The inner sum is also empty if $p\, |\, [d_i, e_i]$ since then
 \[p\, |\, n+h_i - (n+h_m-b_\ell) =
 h_m - h_i - b_\ell\]
which is absurd, since $h_m - h_i - b_\ell$ is bounded and is nonzero by hypothesis.

We may now replace \eqref{eq2.20} by
 \begin{equation}\label{eq2.21}
\Omega(p) = \sideset{}{'}\sum_{\substack{\bs d, \bs e \\ (d_i,p)=(e_i,p)=1\,\forall i}} \l_{\bs d} \l_{\bs e} \Bigg\{\frac Y{p^2W_1 \prod\limits_{i=1}^k [d_i, e_i]} + O\left(V\left(p^2 W_1 \prod\limits_{i=1}^k [d_i, e_i]\right)\right)\Bigg\},
 \end{equation}
where $\sum'$ denotes a summation restricted by: $(d_i, e_j)=1$ whenever $i \ne j$. Expanding the right-hand side of \eqref{eq2.21}, we obtain a main term of the shape estimated in Lemma 2.5 of \cite{PT}. The argument there gives
 \begin{align*}
\sideset{}{'}\sum_{\substack{\bs d, \bs e \\ (d_i,p)=(e_i,p)=1\,\forall i}} \frac{\l_{\bs d} \l_{\bs e}}{\prod\limits_{i=1}^k [d_i, e_i]} = \sideset{}{'}\sum_{\bs d, \bs e} \frac{\l_{\bs d} \l_{\bs e}}{\prod\limits_{i=1}^k [d_i, e_i]} +O\left(\frac{1}{p} \left(\frac{\phi(W)}{W} \mc L\right)^k\right),
\end{align*}
uniformly for $p > D_0$. As already alluded to above in the discussion of $S_1$, the behavior of the main term here can be read out of the proof of \cite[Proposition 1]{bakzhao}. Collecting our estimates, we find that
\[  \sideset{}{'}\sum_{\substack{\bs d, \bs e \\ (d_i,p)=(e_i,p)=1\,\forall i}} \frac{\l_{\bs d} \l_{\bs e}}{\prod\limits_{i=1}^k [d_i, e_i]} = \frac{\phi(W_1)^k}{W_1^k} \, (\log R)^k I_k(F) (1 + o(1)). \]
Clearly this gives
 \begin{multline*} \sum_{D_0 < p \le B} \Omega(p) \ll
 \frac{Y \phi(W_1)^k}{W_1^{k+1}}\, \mc L^k
 \sum_{p > D_0} p^{-2}
 +  (\max_{\bs d} |\lambda_{\bs d}|)^2  \sum_{D_0 <p \le B} \sum_{\ell \le R^2W_1}
 \mu^2(\ell) \tau_{3k}(\ell) V(p^2\ell).\end{multline*}
(We use \eqref{eq2.21} along with a bound for the number of occurrences of $\ell$ as $W_1 \prod\limits_{i=1}^k [d_i, e_i]$.) On an application of \eqref{eq1.5} with $d = p^2$ satisfying \eqref{eq1.4}, we obtain the bound \eqref{eq2.18}.

Let $\sum\limits_{n;\, \eqref{eq2.22}}$ denote a summation over $n$ with
 \begin{equation}\label{eq2.22}
N \le n < N + M, \ n \equiv \nu_0 \ (\mod W_1), \, p^2\,|\, n+h_m-b_\ell \  (\text{some } p > B).
 \end{equation}

Cauchy's inequality gives
 \begin{align*}
\sum_{\substack{n\in \mc A\\
p^2\,|\, n+h_m-b_\ell \ (\text{some } p > B)}} w_n &\le \sum_{n; \, \eqref{eq2.22}} w_n\\
&\le \Bigg(\sum_{n;\, \eqref{eq2.22}} 1\Bigg)^{1/2} \Bigg(\sum_{\substack{n\equiv \nu_0\, (\mod W_1)\\
N \le n < N + M}} w_n^2\Bigg)^{1/2}\\
&\ll \Bigg(\sum_{B < p \le (3N)^{1/2}} \left(\frac M{p^2W_1} + 1\right)\Bigg)^{1/2}\left(\frac{M^{1/2}}{W_1^{1/2}}\, \mc L^{19k/2} + N^\t\right)\\
\intertext{(by \eqref{eq2.17})}
&\ll \frac{M\mc L^{19k/2}}{W_1B^{1/2}} + \frac{N^\t M^{1/2}}{W_1^{1/2} B^{1/2}} + \frac{M^{1/2}N^{1/4} \mc L^{19k/2}}{W_1^{1/2}} + N^{\frac 14 + \t}.
 \end{align*}
To complete the proof we verify (disregarding $W_1$) that each of these four terms is $\ll Y\mc L^{k-1/2}$. We have
 \[M\mc L^{19k/2} B^{-1/2} (Y\mc L^{k-1/2})^{-1} \ll 1\]
since $B \ge \mc L^{18k}(MY^{-1})^2$. We have
 \[N^\t M^{1/2} B^{-1/2} (Y\mc L^{k-1/2})^{-1}\ll 1\]
since $B \ge (MY^{-1})^2$ $N^{2\t} M^{-1}$. We have
 \[M^{1/2} N^{1/4} \mc L^{19k/2}(Y\mc L^{k-1/2})^{-1}\ll 1\]
since $Y \gg N^{1/4} \mc L^{9k} M^{1/2}$. Finally,
 \[N^{1/4+\t}(Y\mc L^{k-1/2})^{-1} \ll 1\]
since $Y \gg N^{\t + 1/4}$. This completes the proof of the first assertion of \Cref{thm1}.

Now suppose $Y > N^{\frac 12 + \e}$. We can replace $B$ by $B_1 : = (MY^{-1})N^\e$ throughout, and at the last
stage of the proof use the bound
 \begin{equation}\label{eq2.23}
\sum_{\substack{p^2\mid n+h_m-b_\ell\\
(\text{some } p > B_1)}} w_n \le w \sum_{\substack{
N\le n \le N + M\\
p^2\mid n+h_m- b_\ell\\
(\text{some } p > B_1)}} 1,
 \end{equation}
where
 \[w := \max_{n} w_n.\]
Now
 \[w = \sum_{[d_i,e_i]\mid n_1+h_i\,\forall i}
 \l_{\bs d} \l_{\bs e}\]
for some choice of $n_1 \le N + M$.
The number of possibilities for $d_1, e_1, \ldots, d_k, e_k$ in this sum is $\ll N^{\e/3}$. Hence \eqref{eq2.23} yields
 \begin{align*}
\sum_{\substack{
p^2\mid n + h_m - b_\ell\\
(\text{some } p > B_1)}}w_n &\ll N^{\e/2} \sum_{B_1 < p \le 3N^{1/2}} \left(\frac M{p^2} + 1\right)\\[2mm]
&\ll \frac{N^{\e/2}M}{B_1} + N^{1/2 + \e/2} \ll Y\mc L^{k-1/2}.
 \end{align*}
The second assertion of \Cref{thm1} follows from this. $\hfill\Box$

\section{Proof of Theorems \ref{thm2} and \ref{thm3}.}

We begin with \Cref{thm3}, taking $b=1$, $\rho(n) = X(\mb P; n)$, $M = Y = N^\phi$, $Y_{m} = \int_N^{N+M} \frac{dt}{\log t}$. By results of Timofeev \cite{tim}, we find that \eqref{eq1.9} holds with $\t = \psi$. The range of $d$ given by \eqref{eq1.4} is
 \begin{equation}\label{eq3.1}
d \ll \mc L^{36k}.
 \end{equation}
Now \eqref{eq1.5} is a consequence of the elementary bound $V(m) \ll 1$.

Turning to the construction of a compatible set $\mc H_k$, let $L = 2(k-1)r + 1$. Take the first $L$ primes $q_1 < \cdots < q_L$ greater than $L$. Select $q_1' = q_1, q_2', \ldots, q_k'$ recursively from $\{q_1, \ldots, q_L\}$ in that $q_j$ satisfies
 \begin{equation}\label{eq3.2}
Pq_j' \ne Pq_i' \pm b_\ell \quad (1 \le i \le j-1,\, 1 \le \ell \le r),
 \end{equation}
a choice which is possible since $L > 2(j-1)r$. Now $\mc H_k = \{Pq_1', \ldots, Pq_k'\}$ is an admissible set compatible with $\mc R$. The set $\mc S$ given by \Cref{thm1} satisfies
 \[D(\mc S) \le P(q_L - q_1) \ll \exp(O(kr)).\]
As for the choice of $k$, the condition \eqref{eq1.10} is satisfied when
 \[k = \left\lceil \exp\left(\frac{2t}\t + C_1\right)\right\rceil \]
because of \eqref{eq1.12}. \Cref{thm3} follows at once.

For \Cref{thm2}, we adapt the proof of \cite[Theorem 3]{bakzhao}. Let $\g = \a^{-1}$, $N = M = v^2$ and $\t = \frac 27 - \e$. We take
 \[\mc A = \{n \in [N, 2N) : n = \lfloor \a m + \b\rfloor
 \text{ for some } m \in \mb N\} \quad\text{and}\quad Y=\g N.\]
We find as in \cite{bakzhao} that
 \[\mc A = \{n \in [N, 2N) : \g n \in I \ (\mod 1)\},\]
where $I = (\g\b - \g, \g\b]$. The properties that we shall enforce in constructing $h_1, \ldots, h_k$ are
 \begin{enumerate}
\item[(i)] $h_1, \ldots, h_k$ is compatible with $\mc R$;

\item[(ii)] we have $h_m = h_m' + h$ $(1 \le m \le k)$, where $h\g \in (\eta - \e\gamma, \eta)$ $(\mod 1)$ and
 \[-\g h_m' \in (\eta, \eta + \e\g) \ (\mod 1)
 \ \text{ for some real } \eta;\]

\item[(iii)] we have
 \[M_k > \frac{2t - 2}{0.90411 \left(\frac 27 - \e\right)}.\]
 \end{enumerate}
The condition (ii) gives us enough information to establish \eqref{eq1.9}; here we follow \cite{bakzhao} verbatim, using the function $\rho = \rho_1 + \rho_2 + \rho_3 - \rho_4 - \rho_5$ in \cite[Lemma 18]{bakzhao}, and taking $b$ slightly larger than 0.90411 in \eqref{eq1.8}.

Turning to \eqref{eq1.5}, with the range of $d$ as in \eqref{eq3.1}, we may deduce this bound from \cite[Lemma 12]{bakzhao} with $M=d$, $a_m = 1$ for $m = d$, $a_m = 0$ otherwise, $Q \le N^{2/7-\e}$, $K = N/d$ and $H = \mc L^{A+1}$. This requires an examination of the reduction to mixed sums in \cite[Section 5]{bakzhao}.

It remains to obtain $h_1, \ldots, h_k$ satisfying (i)--(iii) above. We use the following lemma.

 \begin{lem}\label{lem1}
Let $I$ be an interval of length $\ell$, $0 < \ell < 1$. Let $x_1, \ldots, x_J$ be real and $a_1, \ldots, a_J$ positive.
 \begin{enumerate}
\item[(a)] There exists $z$ such that
 \[\#\{j \le J : x_j \in z + I \ (\mod 1)\} \ge J\ell.\]

\item[(b)] For any $L \in \mb N$, we have
 \[\Bigg|\sum_{\substack{j=1\\
 x_j\in I\ (\mod 1)}}^J a_j - \ell\cdot\sum_{j=1}^{J} a_j\Bigg| \le
 \frac 1{L+1}\sum_{j=1}^{J} a_j + 2 \sum_{m=1}^L \left(
 \frac 1{L+1} + \ell\right)\Bigg|\sum_{j=1}^J a_j
 e(mx_j)\Bigg|.\]
 \end{enumerate}
 \end{lem}

 \begin{proof}
We leave (a) as an exercise. We obtain (b) by a simple modification of the proof of \cite{rcb1}, Theorem 2.1 on revising the upper bound for $|\widehat T_1(m)|$:
 \[|\widehat T_1(m)| \le \frac 1{L+1} +
 \frac{|\sin\pi \ell m|}{\pi m} \le \frac 1{L+1}
 +\ell.\qedhere\]

Now let $\ell$ be the least integer with
 \begin{equation}\label{eq3.3}
\log(\e \g \ell) \ge \frac{2t-2}{0.90411
\left(\frac 27 - \e\right)} + C_1,
 \end{equation}
and let $L = 2(\ell - 1)r + 1$. As above, select primes $q_1', \ldots, q_\ell'$ from $q_1,\ldots, q_L$ so that \eqref{eq3.2} holds. Applying \Cref{lem1}, choose $h_1',\ldots, h_k'$ from $\{Pq_1', \ldots, Pq_\ell'\}$ so that, for some real $\eta$,
 \[-\g h_m' \in (\eta,\eta + \e \gamma) \pmod 1 \quad
 (m=1,\ldots, k)\]
and
 \begin{equation}\label{eq3.4}
k \ge \e\g\ell.
 \end{equation}
We combine \eqref{eq3.3}, \eqref{eq3.4} with \eqref{eq1.12} to obtain (iii). Now there is a bounded $h$, $h\equiv 0\pmod P$, with
 \[\g h\in (\eta - \e\g, \eta) \pmod 1.\]
This follows from \Cref{lem1} with $x_j = jP\g$, since
 \[\sum_{j=1}^J e(m j P\g) \ll
 \frac 1{\| mP\g\|}.\]
We now have (i), (ii) and (iii). \Cref{thm1} yields the required set of primes $\mc S$ with
 \[D(\mc S) \le P(q_L - q_1) \ll \exp(O(\ell r)),\]
and the desired bound \eqref{eq1.13} follows from the choice of $\ell$. This completes the proof of \Cref{thm2}.
 \end{proof}

 \section{Lemmas for the proof of \Cref{thm4}}

We begin by extending a theorem of Robert and Sargos \cite{robsar} (essentially, their result is the case $Q=1$ of \Cref{lem2}).

 \begin{lem}\label{lem2}
Let $H \ge 1$, $N\ge 1$, $M\ge 1$, $Q \ge 1$, $X \gg HN$. For $H < h \le 2H$, $N < n \le 2N$, $M < m \le 2M$ and the characters $\chi\pmod q$, $1 \le q \le Q$, let $a(h, n, q, \chi)$ and $b(m)$ be complex numbers,
 \[|a(h, n, q, \chi)| \le 1, \quad |b(m)| \le 1.\]
Let $\a$, $\b$, $\g$ be fixed real numbers, $\a(\a-1) \b\g \ne 0$. Let
 \[S_0(\chi) = \sum_{H< h \le 2H} \
 \sum_{N < n \le 2N} a(h, n, q, \chi)
 \sum_{M < m \le 2M} b(m) \chi(m) e\left(
 \frac{Xh^\b n^\g m^\a}{H^\b N^\g M^\a}\right).\]
Then
 \begin{multline*}
\sum_{q\le Q} \sum_{\chi \, (\mod q)} |S_0(\chi)| \\
\ll (HMN)^\e \Bigg(Q^2HNM^{\frac 12} + Q^{3/2}HNM  \left(\frac{X^{\frac 14}}{(HN)^{\frac 14}M^{\frac 12}} + \frac 1{(HN)^{\frac 14}}\right)\Bigg).
 \end{multline*}
 \end{lem}

 \begin{proof}
By Cauchy's inequality,
 \begin{multline*}
|S_0(\chi)|^2  \\ \le HN \sum_{H < h \le 2H} \ \sum_{N < n \le 2N}
 \ \sum_{\substack{M < m_1 \le 2M\\
M < m_2 \le 2M}} b(m_1)\overline{b(m_2)} \chi(m_1) \overline{\chi(m_2)} e(Xu (h,n) v(m_1,m_2)),
 \end{multline*}
with
 \[u(h,n) = \frac{h^\b n^\g}{H^\b N^\g}, \quad v(m_1,m_2) =
 \frac{m_1^\a - m_2^\a}{M^\a}.\]
Summing over $\chi$,
 \begin{multline*}
\sum_{\chi \, (\mod q)} |S_0(\chi)|^2 \\ \le HN \sum_{H < h \le 2H} \ \sum_{N < n \le 2N} \phi(q) \sum_{\substack{M < m_1 \le 2M\\
M < m_2 \le 2M\\
m_1\,\equiv\, m_2\, (\mod q)}} b(m_1) \overline{b(m_2)}e(X u(h, n) v(m_1, m_2)).
 \end{multline*}
Separating the contribution from $m_1 = m_2$, and summing over $q$,
 \[\sum_{q\le Q} \ \sum_{\chi \, (\mod q)} |S_0(\chi)|^2
 \le H^2 N^2 M \sum_{q\le Q} \phi(q) + S_1,\]
where
\[
S_1 = C(\e) M^\e QHN \sum_{H<h \le 2H} \ \sum_{N < n \le 2N}
\sum_{\substack{
 M < m_1 \le 2M\\
 M < m_2 \le 2M}} w(m_1,m_2) e(Xu(h,n) v(m_1,m_2)),
\]
with
 \[w(m_1,m_2) = \begin{cases}
 0 & \text{if } m_1=m_2,\\
 \ds\sum_{q \le Q} \ \sum_{m_1-m_2 = qn,\, n \in \mb Z}
 \, \frac{b(m_1)\overline{b(m_2)} \phi(q)}
 {C(\e) M^{\rz{\e}} Q} & \text{if } m_1\ne m_2.
 \end{cases}\]
Note that
 \[|w(m_1, m_2)| \le 1\]
for all $m_1$, $m_2$ if $C(\e)$ is suitably chosen.

We now apply the double large sieve to $S_1$ exactly as in \cite[(6.5)]{robsar}. Using the upper bounds given in \cite{robsar}, we have
 \[S_1 \ll M^\e QHNX^{1/2} \mc B_1^{1/2}
 \mc B_2^{1/2},\]
where
 \begin{align*}
\mc B_1 = \sum_{\substack{h_1, n_1, h_2, n_2\\
|u(h_1, n_1) - u(h_2, n_2)|\le 1/X\\
H < h_i \le 2H, \, N < n_i \le 2N\ (i=1,2)}} 1 &\ll (HN)^{2+\e} \left(\frac 1{HN} + \frac 1X\right)\\
&\ll (HN)^{1+\e},\\
\intertext{and}
\mc B_2 = \sum_{\substack{m_1, m_2, m_3, m_4\\
|v(m_1, m_2) - v(m_3, m_4)|\le 1/X\\
M < m_i \le 2M \ (1 \le i \le 4)}} 1 &\ll M^{4+\e} \left(\frac 1{M^2} + \frac 1X\right).
 \end{align*}
Hence
\[
\sum_{q \le Q} \ \sum_{\chi \, (\mod q)} |S_0(\chi)|^2 \ll Q^2 H^2 N^2 M + (MHN)^{2+2\e}Q \left(\frac{X^{1/2}}{(HNM^2)^{1/2}} + \frac 1{(HN)^{1/2}}\right).
\]
\Cref{lem2} follows on an application of Cauchy's inequality.
 \end{proof}

 \begin{lem}\label{lem3}
Fix $c$, $0 < c < 1$. Let $h \ge 1$, $m \ge 1$, $K > 1$, $K' \le 2K$,
 \[S = \sum_{K < k \le K',\, mk \equiv u\ (\mod q)}
 e(h(mk)^c).\]
Then for any $q$, $u$,
 \[S \ll (hm^c K^c)^{1/2} + K(hm^c K^c)^{-1/2}. \]
 \end{lem}

 \begin{proof}
We write $S$ in the form
 \begin{align*}
S &= \frac 1q \sum_{K < k \le K'} \ \sum_{r=1}^q e \left(\frac{r(mk-u)}q + h(mk)^c\right)\\
&= \frac 1q \sum_{r=1}^q e\left(-\frac{ur}q\right) \sum_{K < k \le K'} e\left(\frac{rmk}q + h(mk)^c\right),
 \end{align*}
and apply \cite[Theorem 2.2]{grakol} to each sum over $k$.
 \end{proof}

\section{Proof of \Cref{thm4}}

Throughout this section, fix $c \in \left(\frac 89, 1\right)$ and define, for an interval $I$ of length $|I|<1$,
\[\mc A(I) = \{n \in [N, 2N) : n^c \in I\pmod 1\}.\]
We choose $\mc H_k$ compatible with $\mc R$ as in the proof of \Cref{thm3}, so that
 \[h_k - h_1 \ll \exp(O(kr)).\]
We apply the second assertion of \Cref{thm1} with
 \[M= N, \quad Y = N^{c+\e}, \quad b= 1,\quad
 \rho(n) = X(\mb P; n).\]
We define $\t$ by
 \[\t = \frac{9c-8}6 - \e, \]
and we choose $k = \lceil\exp(\frac{2t-2}{\t} + C_1)\rceil$, so that \eqref{eq1.10} holds.
By our choice of $\theta$, the range in \eqref{eq1.11} is contained in
 \begin{equation}\label{eq5.1}
1 \le d \le N^{2-2c}.
 \end{equation}
It remains to verify \eqref{eq1.5} and \eqref{eq1.9} for a fixed $h_m$. We consider \eqref{eq1.9} first.

The set $(\mc A + h_m) \cap \mc A$ consists of those $n$ in $[N, 2N)$ with
 \[n^c-\b \in [0, N^{-1+c+\e})\ (\mod 1), \
 (n+h_m)^c - \b \in [0, N^{-1+c+\e})\ (\mod 1).\]
Since
 \[(n + h_m)^c = n^c + O(N^{c-1}) \quad
 (N \le n < 2N),\]
we have
 \begin{equation}\label{eq5.2}
\mc A(I_2) \subset (\mc A + h_m) \cap \mc A \subset \mc A(I_1)
 \end{equation}
where, for a given $A$,
 \begin{align*}
I_1 &= [\b, \b + N^{-1+c+\e}),\\
I_2 &= [\b, \b + N^{-1 + c + \e} \ (1 - \mc L^{-A-3k})).
 \end{align*}
By a standard partial summation argument it will suffice to show that, for any choice of $u_q$ relatively prime to $q$,
 \[\sum_{q\le N^\t} \mu^2(q) \tau_{3k}(q) \Bigg|
 \sum_{\substack{n\, \equiv\, u_q\ (\mod q)\\
 N \le n < N'}} \left(\Lambda(n) X((\mc A +h_m)\cap \mc A; n)
 -N^{-1+c+\e}\frac{q}{\phi(q)}\right)\Bigg| \ll Y\mc L^{-A}\]
for $N' \in [N, 2N)$. 
(The implied constant here and below may depend on $A$.) In view of \eqref{eq5.2}, we need only show that for any $A> 0$,
 \begin{equation}\label{eq5.3}
\sum_{q\le N^\t} \mu^2(q) \tau_{3k}(q) \Bigg|\sum_{\substack{
 n\, \equiv\, u_q\ (\mod q)\\
 N \le n < N'}} \left(\Lambda(n) X(\mc A (I_j); n)
 -N^{-1+c+\e}\frac{q}{\phi(q)}\right)\Bigg| \ll Y\mc L^{-A} \ (j=1, 2).
 \end{equation}
The sum in \eqref{eq5.3} is bounded by $\sum_1 + \sum_2$, where
 \[\sum\nolimits_1 = \sum_{q\le N^\t} \mu^2(q)\tau_{3k}(q) \Bigg|
 \sum_{\substack{n\, \equiv\, u_q\ (\mod q)\\
 n^c \in I_j \ (\mod 1)\\
 N \le n < N'}} \Lambda(n) - N^{-1+c+\e} \sum_{\substack{
 n\,\equiv\, u_q\ (\mod q)\\
 N \le n < N'}}\Lambda(n)\Bigg|\]
and
 \[\sum\nolimits_2 = N^{-1+c+\e} \sum_{q\le N^\t} \mu^2(q) \tau_{3k}(q) \Bigg|
 \sum_{\substack{
 n\,\equiv\, u_q\ (\mod q)\\
 N \le n < N'}} \bigg(\Lambda(n) - \frac{q}{\phi(q)}\bigg)\Bigg|.\]
Deploying the Cauchy-Schwarz inequality in the same way as in \cite[(5.20)]{may}, it follows from the Bombieri-Vinogradov theorem that
 \[\sum\nolimits_2 \ll N^{c+\e} \mc L^{-A}.\]
Moreover,
 \[\sum_{q\le N^\t} \mu^2(q) \tau_{3k}(q) \Bigg|N^{-1+c+\e}
 \sum_{\substack{n\, \equiv\, u_q\ (\mod q) \\ N \le n < N'}} \Lambda(n) -
 |I_j| \sum_{\substack{n\, \equiv\, u_q\ (\mod q) \\ N \le n < N'}}\Lambda(n)\Bigg|
 \ll N^{c+\e} \mc L^{-A}\]
(trivially for $j=1$, and by the Brun-Titchmarsh inequality for $j=2$). Thus it remains to show that
 \[\sum_{q\le N^\t} \mu^2(q) \tau_{3k}(q)\Bigg|
 \sum_{\substack{n\, \equiv\, u_q\ (\mod q)\\
 n^c \in I_j \ (\mod 1)\\
 N \le n < N'}} \Lambda(n) - |I_j| \sum_{\substack{
 n\, \equiv\, u_q\ (\mod q)\\
 N \le n < N'}} \Lambda(n)\Bigg| \ll N^{c+\e} \mc L^{-A}.\]
Let $H = N^{1-c-\e}\mc L^{A+3k}$.
We apply \Cref{lem1}, with $a_j = \Lambda(N+j-1)$ for $N+j-1\equiv u_q$ $(\mod q)$ and $a_j = 0$ otherwise, and $L=H$. Using the Brun-Titchmarsh inequality, we find that
\begin{multline*} \Bigg|
 \sum_{\substack{n\, \equiv\, u_q\ (\mod q)\\
 n^c \in I_j \ (\mod 1)\\
 N \le n < N'}} \Lambda(n) - |I_j| \sum_{\substack{
 n\, \equiv\, u_q\ (\mod q)\\
 N \le n < N'}} \Lambda(n)\Bigg| \\
  \ll \frac{N^{c+\e}}{\phi(q)}\mc L^{-A-3k} + N^{-1+c+\e}\sum_{1 \le h \le H} \Bigg| \sum_{\substack{
 N \le n < N'\\
 n\, \equiv\, u_q \, (\mod q)}} \Lambda(n) e(hn^c) \Bigg|.\end{multline*}
Recalling the upper estimate $\tau_{3k}(q) \ll    N^{\e/20}$ for $q \le N^{\t}$, it suffices to show that
\[\sum_{q\le N^\t} \sum_{1 \le h \le H}\,
 \s_{q,h} \, \sum_{\substack{
 N \le n < N'\\
 n\, \equiv\, u_q \, (\mod q)}} \Lambda(n) e(hn^c)
 \ll N^{1-\e/10}\]
for complex numbers $\s_{q,h}$ with $|\s_{q,h}|\le 1$.

We apply a standard dyadic dissection argument, finding that it suffices to show that
 \begin{equation}\label{eq5.4}
\sum_{q\le N^\t} \sum_{H_1\le h \le 2H_1} \s_{q,h}
\sum_{\substack{
N \le n < N'\\
n\, \equiv \, u_q\, (\mod q)}} \Lambda(n) e(hn^c) \ll N^{1-\e/9}
 \end{equation}
for $1 \le H_1 \le H$. The next step is a standard decomposition of the von Mangoldt function; see for example \cite[Section 24]{davenport}. In order to obtain \eqref{eq5.4}, it suffices to show that
 \begin{equation}\label{eq5.5}
\sum_{q\le N^\t} \sum_{H_1 \le h \le 2H_1} \s_{q,h} \mathop{\sum_{M \le m < 2M}\ \sum_{K \le k < 2K}}_{\substack{N \le mk < N'\\
mk\, \equiv\, u_q\, (\mod q)}} a_m b_k e(h(mk)^c) \ll N^{1-\e/8}
 \end{equation}
for complex numbers $a_m$, $b_k$ with $|a_m|\le 1, |b_k| \le 1$, subject to {\it either}
\begin{align}
N^{1/2} &\ll K \ll N^{2/3}\label{eq5.6}\end{align}
\emph{or}
\begin{align}
K &\gg N^{2/3}, \quad b_k = \begin{cases}
 1 & \text{if }K \le k < K',\\
 0 & \text{if }K' \le k < 2K.
 \end{cases} \label{eq5.7}
 \end{align}

We first obtain \eqref{eq5.5} under the condition \eqref{eq5.6}. We replace \eqref{eq5.5} by
\begin{multline*}
\sum_{q\le N^\t} \ \frac{1}{\phi(q)} \  \sum_{\chi \, (\mod q)} \overline{\chi}(u_q) \sum_{H_1 \le h_1 \le 2H_1} \s_{q,h} \mathop{\sum_{M \le m < 2M}\ \sum_{K \le k < 2K}}_{N \le mk < N'} a_mb_k \chi(m)\chi(k)e(h(mk)^c) \\ \ll N^{1-\e/8}.
 \end{multline*}
A further dyadic dissection argument reduces our task to showing that
 \begin{equation}\label{eq5.8}
\sum_{Q \le q \le 2Q} \ \sum_{\chi\, (\mod q)} \Bigg|
\sum_{H_1 \le h \le 2H_1} \s_{q,h} \
\sum_{M \le m < 2M}\ \sum_{K \le k < 2K} a_mb_k \chi(m) \chi(k) e(h(mk)^c)\Bigg| \ll QN^{1-\e/7}
 \end{equation}
for $Q < N^{\theta}$.

We now apply \Cref{lem2} with $X = H_1N^c$ and $(H_1, K, M)$ in place of $(H, N, M)$. The condition $X \gg H_1K$ follows easily since $K \ll N^c$. Thus the left-hand side of \eqref{eq5.8} is
 \begin{align*}
&\ll (H_1 N)^{\e/8} (Q^2 H_1 N^{1/2} K^{1/2} + Q^{3/2} H_1 N^{\frac 12 + \frac c4} K^{1/4} + Q^{3/2} H_1^{3/4} NK^{-1/4})\\[2mm]
&\ll N^{\e/7}(Q^2H_1 N^{5/6} + Q^{3/2} H_1 N^{2/3+c/4} + Q^{3/2}H_1^{3/4} N^{7/8})
 \end{align*}
using \eqref{eq5.6}. Each term in the last expression is $\ll QN^{1-\e/7}$:
\begin{align*}
N^{\e/7}Q^2 H_1 N^{5/6} (QN^{1-\e/7})^{-1} &\ll N^{\t+5/6 -c + 2\e/7} \ll 1,\\
N^{\e/7}Q^{3/2} H_1N^{2/3+c/4}(QN^{1-\e/7})^{-1} &\ll N^{\t/2 + 2/3-3c/4+2\e/7} \ll 1,\\
N^{\e/7}Q^{3/2}H_1^{3/4}N^{7/8}(QN^{1-\e/7})^{-1} &\ll N^{\t/2 + 5/8 - 3c/4 +2\e/7} \ll 1.
 \end{align*}

We now obtain \eqref{eq5.5} under the condition \eqref{eq5.7}. By \Cref{lem3}, the left-hand side of \eqref{eq5.5} is
 \begin{align*}
&\ll N^{\theta} MH_1((H_1N^c)^{1/2} + K(H_1N^c)^{-1/2})\\[2mm]
&\ll H_1^{3/2} N^{1+c/2 + \theta} K^{-1} + H_1^{1/2} N^{1-c/2+\theta}\\[2mm]
&\ll N^{11/6-c+\t} + N^{3/2 - c + \t} \ll N^{1 - \e/8}.
 \end{align*}

Turning to \eqref{eq1.5}, (under the condition \eqref{eq1.11} on $d$) by a similar argument to that leading to \eqref{eq5.5}, it suffices to show that
 \begin{equation}\label{eq5.9}
\sum_{\substack{q \le N^\t\\
(q,d) = 1}} \sum_{H_1\le h \le 2H_1} \Bigg|  \sum_{\substack{
N\le n \le N'\\
n\equiv u_{qd}\, (\mod qd)}} e(hn^c)\Bigg|\ll N^{1-\e/3}d^{-1}
 \end{equation}
for $d \le N^{2-2c}$, $H_1 \le N^{1-c}$, $N \le N' \le 2N$. By \Cref{lem3}, the left-hand side of \eqref{eq5.9} is
 \[\ll N^{\theta} H_1((H_1N^c)^{1/2} + N(H_1N^c)^{-1/2}).\]
Each of the two terms here is $\ll N^{1-\e/3} d^{-1}$. To see this,
 \begin{align*}
N^{\theta} H_1^{3/2}N^{c/2}(N^{1-\e/3}d^{-1})^{-1} &\ll N^{\t + 1/2-c} N^{2-2c} \ll 1\\
\intertext{and}
N^{\theta} H_1^{1/2} N^{1-c/2}(N^{1-\e/3}d^{-1})^{-1} &\ll N^{\t + 1/2-c} N^{2-2c} \ll 1.
 \end{align*}
This completes the proof of \Cref{thm4}. $\hfill\Box$

\subsection*{Acknowledgments} The second author is supported by NSF award DMS-1402268. This work began while the second author was visiting BYU. He thanks the BYU mathematics department for their hospitality.

 \end{document}